\documentclass[a4paper]{article} 

\usepackage{amsmath,amsfonts,amssymb,amsthm,mathtools} 
\usepackage{icomma} 

\usepackage{euscript}	
\usepackage{mathrsfs} 

\usepackage{euscript}
\usepackage{mathrsfs}

\usepackage{setspace}

\theoremstyle{plain}
\newtheorem{theorem}{Theorem}
\newtheorem{lemma}{Lemma}

\theoremstyle{definition}
\newtheorem{definition}{Definition}

\let\rom\textup

\let\qq\qquad
\let\q\quad

\usepackage{geometry}
\usepackage{setspace}
\begin{document}

\title{%
	Solution of the Unconditional Extremal Problem\\
	for a Linear-Fractional Integral Functional\\
	Depending on the Parameter}
\author{
	P.~V.~Shnurkov
	\footnote{National Research University Higher School of Economics, 
		Moscow, 101000 Russia, 
		pshnurkov@hse.ru}, 
	K.~A.~Adamova
	\footnote{National Research University Higher School of Economics, 
		Moscow, 101000 Russia, 
		kaadamova@edu.hse.ru}
	}
\date{}

\maketitle

\markboth{Shnurkov, Adamova}{Extremal Problem for a Linear-Fractional Integral Functional}

\begin{abstract}
 The paper is devoted to the study of the unconditional extremal problem for a fractional linear
integral functional defined on a set of probability distributions.
 In contrast to results proved earlier,
the integrands of the integral expressions
in the numerator and the denominator
in the problem under consideration
depend on a real optimization parameter vector.
 Thus, the optimization problem is studied
on the Cartesian product of a set of probability distributions
and a set of admissible values of a real parameter vector.
 Three statements on the extremum of a fractional linear integral functional are proved.
 It is established that, in all the variants, the solution of the original problem
is completely determined by the extremal properties of the test function
of the linear-fractional integral functional; this function
is the ratio of the integrands of the numerator and the denominator. 
Possible applications of the results obtained
to problems of optimal control of stochastic systems are described.
\end{abstract}


\section{Introduction}

 The present paper is a continuation of the studies of the unconditional extremal problem
for linear-fractional integral functionals
carried out in the papers~[1],~[2].
 Besides its theoretical value,
this problem serves as a basis for solving optimal control problems
for various classes of random processes
(regenerating, Markov, semi-Markov).
 In turn, optimal control problems for these classes of processes
arise in the analysis of numerous applied models
in the mathematical theory of effectiveness
and reliability, storage theory,
and other areas of applied probability theory.

 Let us make some remarks of bibliographic nature
concerning the general unconditional extremal problem
for linear-fractional integral functionals.

 By linear-fractional programming
we usually mean the area of optimization theory
in which the objective functional of the extremal problem under consideration
is the ratio of two linear functionals
and the available constraints are linear.
 In this area of optimization theory,
there exists an extensive literature
mostly devoted to the study of such problems
in finite-dimensional spaces.

 A comprehensive theory dealing with this direction
was described in the fundamental monograph~[3].
 This book not only presents theoretical results
concerning the solutions of the corresponding extremal problems,
but also describes numerical methods for finding such solutions.
 In addition,
it contains a detailed bibliography
in the field of fractional linear programming.
 Also note certain recent important papers, such as [4], [5], [6],
in which theoretical and numerical problems related to research in this subject
were studied.

 A special area of linear-fractional programming
comprises extremal problems in which the objective functional
is the ratio of two integrals.
 The integrands in these integrals
are assumed to be known
and integration is carried out with respect to a probability measure
belonging to a set of probability measures
defined on a given measure space.
 The solution of the problem is the probability measure
furnishing a global extremum to such a functional.
 Functionals of this form may be called \textit{integral linear-fractional functionals}.
 Extremal problems for integral linear-fractional objective functionals
defined on a set of probability distributions
in a finite-dimensional space were considered by V.~A.~Kashtanov
in~[7],~[8].
 His results on the theory of unconditional extrema
for such functionals were described
in complete form in the monograph~[8, Chap.~10].
 The main significant feature of these results is that
the unconditional extremum of functionals of such form
is attained on degenerate probability distributions
having one point of growth.
 However, this result was obtained in [8]
under highly restrictive conditions, the chief of which
was the assumption on the existence of an extremum of the objective functional,
i.e., the existence of a solution of the original problem.
 In his papers~[1],~[2],
 P.~V.~Shnurkov constructed a new solution of the unconditional extremal problem
for a linear-fractional integral functional,
which significantly generalizes
and strengthens the corresponding results from [8].
 The fundamental difference of the results given in [1],~[2] from
the earlier ones is that the main statement indicates conditions
under which the extremum of a fractional linear integral functional
exists and is attained on a degenerate distribution concentrated
at one point.
 Further, the point at which the whole probability measure
is concentrated is the point of global extremum of a function
for which an explicit analytic representation
was obtained.

 In the papers~[1],~[2],
it was assumed that the integrands of the numerator
and denominator of the fractional linear integral functional under consideration
are independent of the probability measure
characterizing the control.
 This assumption is justified by the fact that, in
many applied problems, the objective exponent
is a stationary cost functional with a prescribed structure.
 However,
in many specific problems,
the objective exponent is expressed as
a linear-fractional integral functional
the integrands of whose numerator and denominator
depend on a collection of deterministic control parameters.
 In this case, the extremal problem is changed,
and a special study is needed.
 Such a study is given in the present paper.

\section{Statement of the Main Extremal Problem}
 Let
$(U, \mathscr{B})$
denote a measure space,
where
$U$ is an arbitrary set and
$\mathscr{B}$
is the
$\sigma$-algebra of subsets of the set~$U$
including all simpletons.

 In what follows, the space~$(U, B)$
will serve as a set of admissible stochastic
optimization or of control parameters
in the extremal problem under consideration.

 Let $\Gamma$
be a set of probability measures defined on the
$\sigma$-algebra~$\mathscr{B}$
whose elements will be denoted by
$\Psi \in \Gamma$.
 In what follows, we shall formulate constraints
on the set~$\Gamma$ and its elements related to the main extremal problem.

\begin{definition}
 A probability measure~$\Psi^*$
defined on~$\mathscr{B}$,
is said to be \textit{degenerate} if
there exists a point
$u^* \in U$
such that
$\Psi^*(\{u^*\})=1$,
$\Psi^*(B^*)=0$,
where
$u^*=\{u^* \}$
is the set consisting of one point
and $B^*$
is an arbitrary set from the system~$\mathscr{B}$
not containing the point~$u^*$.
 The point
$u^*$
is called the \textit{concentration point} of the measure~$\Psi^*$,
and we will denote the measure
with its concentration point by the symbol
$\Psi^*=\Psi^*_{u^*}$.
\end{definition}

 Let
$\Gamma^*$
denote the set of all possible degenerate probability measures
defined on the measure space~$(U, \mathscr{B})$.
 The set $\Gamma^*$ is in a one-to-one correspondence with
with the set of points concentration of degenerate
probability measures~$U$.

 Let
$S\subseteq R^r $
be a set of values of the vector parameter
$\alpha=(\alpha_1, \alpha_2, \ldots , \alpha_r) \in S.$
 In what follows, the parameter~$\alpha$
will serve as a deterministic optimization (control) parameter
in the optimization problem under consideration.

 Let us define
some measurable numerical functions:
$$ 
\text{
$A(\alpha, u):S \times
U \rightarrow R$,\qquad
$B(\alpha, u):S \times
U \rightarrow R$,
}
$$ 
where
$u \in U$,
$\alpha \in S$.

 We introduce the following integral transformations
defined by the functions
$A(\alpha, u)$,
$B(\alpha,u)$:
$$
I_{1,\alpha} (\Psi) = \int\limits_U A(\alpha, u)\,d\Psi(u), \qquad
I_{2,\alpha} (\Psi) = \int\limits_U B(\alpha, u)\,d\Psi(u).
\eqno(1)
$$

 The integral expressions in relations~(1)
are Lebesgue integrals with respect to the probability measure
$\Psi \in \Gamma$.
 According to the general probability model
examined in great detail in [9, Chap.~2],
these integrals have the meaning of expectations of some
random variables
$A(\alpha, u)$,
$B(\alpha, u)$
depending on an elementary outcome
$u \in U $ of a random experiment
on the probability space~$(U,\mathscr{B},\Psi)$.
 If the set~$U$ is a finite-dimensional real space, then the probability measure~$\Psi$
can be given by the distribution function of a one-dimensional or multidimensional random variable.
 In that case, the integrals in relations~(1)
can be expressed as a Lebesgue-Stieltjes integral with respect to the probability distribution~$\Psi$ [9], [10].
 The integral transformations~(1) also define the following functionals
given on the set of probability measures~$\Gamma$,
$$
I_{1,\alpha} (\Psi):\Gamma \rightarrow R,\qq I_{2,\alpha} (\Psi):\Gamma \rightarrow R,
$$
these functionals depend on the parameter
$\alpha \in S$.

 Let us now introduce the notion of a fractional linear integral functional
depending on a parameter.

\begin{definition}
 The mapping
$I_\alpha: \Gamma \rightarrow R$
defined by the relation
$$
I_\alpha (\Psi) = \frac{I_{1,\alpha} (\Psi)}{I_{2,\alpha} (\Psi)}=\frac{\int\limits_U
A(\alpha, u)\,d\Psi(u)}{\int\limits_U B(\alpha, u)\,d\Psi(u)}
\eqno(2)
$$
will be called a \textit{linear-fractional integral functional
depending on a parameter}~$\alpha$.
\end{definition}

 We shall consider the extremal problem
$$
I_\alpha (\Psi) \rightarrow \operatorname{extr}, \qquad \Psi \in \Gamma,\q \alpha \in S,
\eqno(3)
$$
for a linear-fractional integral functional
depending on a numerical parameter~$\alpha$.

\begin{definition}
 The function
$$ 
\text{
$C(\alpha, u)=\frac{A(\alpha, u)}{B(\alpha, u)},\qq
u \in U,\quad \alpha \in S, $
}
$$ 
is called the \textit{test function of the linear-fractional integral functional}~(2).
\end{definition}

 The unconditional extremal problem
for a linear-fractional integral functional of the form~(2) do independent of the parameter
was studied in the author's papers~[1],~[2].
 However, problem~(3) considered
in the present paper cannot be reduced
to the problem examined in those papers.
 Namely, the integrands of the numerator
and denominator of the objective functional~(2)
depend on an additional nonrandom
control parameter $\alpha \in S$.
 In the papers
[1],~[2],
it was assumed that the integrands of the numerator
and denominator of the fractional linear integral functional
are independent of the control.

 We shall introduce preliminary conditions
on the main objects appearing in the description of the extremal problem~(3).
 These conditions ensure the well-posedness
of the problem under consideration.
 They are as follows:
\begin{enumerate}
\item
 The integral expressions
$$
I_{1,\alpha} (\Psi) = \int\limits_U A(\alpha, u)\,d\Psi(u),\qq
I_{2,\alpha} (\Psi) = \int\limits_U B(\alpha, u)\,d\Psi(u)
$$
exist for all
$\Psi \in \Gamma$, $\alpha \in S$.
\item $\int\limits_U B(\alpha, u)\,d\Psi(u) \neq 0$, $\Psi \in \Gamma$, $\alpha \in S$.
\item $\Gamma^* \subset \Gamma$.
\end{enumerate}
\medskip

\noindent\textbf{Remark.}
 If the function
$B(\alpha, u)$
does not change sign,
i.e.,
$B(\alpha, u) > 0, u \in U,\quad \alpha \in S$,
or
$B(\alpha, u) < 0, u \in U,\quad \alpha \in S$,
then Condition~2 from the system of preliminary conditions
holds automatically.
 At the same time, the strict positivity condition for the function
$B(\alpha, u)$
is natural for many optimal control problems for
regenerating and semi-Markov random processes
(see the corresponding remarks~[2]).
 In this connection,
in the following main statement
on the extremum of the fractional linear integral functional,
it will be assumed that conditions~1,~3
and the strict sign-constancy condition for the function~$B(\alpha, u)$
hold.

\begin{theorem}
 Suppose that the main objects in the extremal problem~\rom{(3)}
satisfy the preliminary conditions~\rom{1,~3}
and the function
$B(\alpha, u)$
is strictly of constant sign
\rom(strictly positive or
strictly negative\rom)
for all
$u \in U,\quad \alpha \in S$.
 We also assume that the test function
$C(\alpha, u)$
attains its global extremum
on the whole set
$(\alpha, u) \in S \times
U $
at the point
$(\alpha^*,u^*)$.

 Then the solution of the extremal problem~(3)
exists and is attained on the pair
$(\alpha^*,\Psi^*_{u^*})$,
where
$\Psi^*_{u^*}$
is a degenerate probability measure
concentrated at the point~$u^*$,
and the following relations hold\rom:
$$
\max_{(\alpha, \Psi) \in S \times
\Gamma} I_\alpha(\Psi) = \max_{\alpha \in S}
\max_{\Psi^* \in \Gamma^*} I_\alpha(\Psi^*) = \max_{(\alpha, u) \in S \times
U}
\frac{A(\alpha, u)}{B(\alpha, u)} = \frac{A(\alpha^*, u^*)}{B(\alpha^*, u^*)},
$$
if
$(\alpha^*, u^*)$
is the point of global maximum of the function
$C(\alpha^*, u^*)$\rom;
$$
\min_{(\alpha, \Psi) \in S \times
\Gamma} I_\alpha(\Psi) = \min_{\alpha \in S}
\min_{\Psi^* \in \Gamma^*} I_\alpha(\Psi^*) = \min_{(\alpha, u) \in S \times
U}
\frac{A(\alpha, u)}{B(\alpha, u)} = \frac{A(\alpha^*, u^*)}{B(\alpha^*, u^*)},
$$
if
$(\alpha^*, u^*)$
is the point of global minimum of the function
$C(\alpha^*, u^*)$.
\end{theorem}

 As a preliminary,
we shall prove the following lemma.

\begin{lemma}
 If the assumptions of the theorem are satisfied
and the function
$$
C(\alpha, u)=\frac{A(\alpha,u)}{B(\alpha, u)}
$$
is bounded above or below,
i.e., at least one of the following inequalities holds\rom:
\begin{align}
 C(\alpha, u) &\leq C^{(+)}, \qquad (\alpha, u) \in S \times
U,
\tag{4}
\\
 C(\alpha, u) &\geq C^{(-)}, \qquad (\alpha, u) \in S \times
U,
\tag{5}
\end{align}
then the corresponding estimate also holds
for the whole linear-fractional functional\rom:
\begin{align}
 I_\alpha (\Psi) &\leq C^{(+)}, \qquad \Psi \in \Gamma,\q \alpha \in S,
\tag{6}
\\
 I_\alpha (\Psi) &\geq C^{(-)}, \qquad \Psi \in \Gamma,\q \alpha \in S.
\tag{7}
\end{align}
\end{lemma}

\begin{proof}[\bf Proof of Lemma~1] 
 Suppose that inequality~(4) holds:
$$
C(\alpha, u)=\frac{A(\alpha, u)}{B(\alpha, u)} \leq C^{(+)} < \infty, \qquad
(\alpha, u) \in S \times
U.
\eqno(8)
$$

 First, consider the variant
in which $B(\alpha, u) > 0$, $(\alpha, u) \in S \times
U$.
 Then, from (8), we obtain
$$
A(\alpha, u) \leq C^{(+)}B(\alpha, u), \qquad (\alpha, u) \in S \times
U.
\eqno(9)
$$

 By the property of the integral~[10],
it follows from (9) that
$$
\int\limits_U A(\alpha, u)\,d\Psi(u) \leq C^{(+)}\int\limits_U B(\alpha, u)\,d\Psi(u),
\qquad \Psi \in \Gamma,\q \alpha \in S.
\eqno(10)
$$

 At the same time,
the strict positivity condition for the function
$B(\alpha, u)$
implies the corresponding inequality for the integral~[10]:
$$
\int\limits_U B(\alpha, u)\,d\Psi(u) > 0, \qquad \Psi \in \Gamma,\q \alpha \in S.
\eqno(11)
$$
 But, in that case, from (10) and (11), we obtain
$$
I_\alpha(\Psi)= \frac{\int\limits_U A(\alpha, u)\,d\Psi(u)}{\int\limits_U B(\alpha,
u)\,d\Psi(u)} \leq C^{(+)}, \qquad \Psi \in \Gamma,\q \alpha \in S.
\eqno(12)
$$

 Now we consider the variant
in which $B(\alpha, u) < 0$, $(\alpha, u) \in S\times U$.

 Then,
using~(8), we can write
$$
A(\alpha, u) \geq C^{(+)}B(\alpha, u), \qquad (\alpha, u) \in S \times
U.
\eqno(13)
$$

 It follows from inequality (13)
and the property of the integral that the following inequality holds:
$$
\int\limits_U A(\alpha, u)\,d\Psi(u) \geq C^{(+)}\int\limits_U B(\alpha, u)\,d\Psi(u),
\qquad \Psi \in \Gamma,\q
\alpha \in S.
\eqno(14)
$$

 At the same time,
the condition of strict negativity of the function~$B(\alpha, u)$
implies the corresponding inequality for the integral:
$$
\int\limits_U B(\alpha, u)\,d\Psi(u) < 0, \qquad \Psi \in \Gamma,\q \alpha \in S.
\eqno(15)
$$

 But, in that case,
from the inequalities~(14) and (15), we obtain
$$
I_\alpha(\Psi)= \frac{\int\limits_U A(\alpha, u)\,d\Psi(u)}{\int\limits_U B(\alpha,
u)\,d\Psi(u)} \leq C^{(+)}, \qquad \Psi \in \Gamma,\q \alpha \in S.
\eqno(16)
$$

 Thus, in both cases in which the function
$B(\alpha, u)$ does not change sign, using condition (8),
we obtain inequality~(6) expressed as~(12) and (16).

 The first assertion of Lemma~1 is proved.

 The second assertion of the lemma (inequality~(7))
is proved in a similar way, .

 The proof of Lemma~1 is complete.
\end{proof}

 Let us pass to the proof of the main assertion of the theorem.

\begin{proof}[ \bf Proof of Theorem~1]
 Suppose that the test function
$C(\alpha, u)$
attains its global maximum
at the point
$(\alpha^*, u^*) \in S \times
U$.
 This assumption implies the estimate
$$
C(\alpha, u) = \frac{A(\alpha, u)}{B(\alpha, u)} \leq \frac{A(\alpha^*,
u^*)}{B(\alpha^*, u^*)}, \qquad (\alpha, u) \in S \times
U.
\eqno(17)
$$
 Then the function
$C(\alpha, u)$
satisfies the assumptions of Lemma~1.
 Applying this lemma,
we obtain the inequality of the form~(6):
$$
I_\alpha(\Psi)= \frac{\int\limits_U A(\alpha, u)\,d\Psi(u)}{\int\limits_U B(\alpha,
u)\,d\Psi(u)} \leq \frac{A(\alpha^*, u^*)}{B(\alpha^*, u^*)}, \qquad \Psi \in \Gamma,\q \alpha
\in S.
\eqno(18)
$$

 Consider the degenerate probability measure
$\Psi^*_{u^*}$,
concentrated at the point~$u^*$.
 By the property of the integral,
for any fixed value of the parameter
$\alpha \in S$,
the following equalities hold:
$$
\int\limits_U A(\alpha, u)\,d\Psi^*_{u^*}(u) = A(\alpha, u^*),\qquad
\int\limits_U B(\alpha, u)\,d\Psi^*_{u^*}(u) = B(\alpha, u^*),
$$
whence
$$
I_\alpha(\Psi^*_{u^*})= \frac{A(\alpha, u^*)}{B(\alpha, u^*)}, \qquad \alpha \in S.
\eqno(19)
$$

 By assumption, the pair
$(\alpha^*, u^*)$
is the point of global maximum of the function
$C(\alpha, u) = {A(\alpha, u)}/{B(\alpha, u)}$.
 Hence,
in view of (19),
it follows that, for
any fixed value of
$\alpha \in S$,
$$
I_\alpha(\Psi^*_{u^*})= \frac{A(\alpha, u^*)}{B(\alpha, u^*)} \leq
\frac{A(\alpha^*, u^*)}{B(\alpha^*, u^*)}, \qquad \alpha \in S.
\eqno(20)
$$

 It follows from inequality (18)
that the set of values of the functional
$I_\alpha(\Psi)$
is bounded above
for all
$\alpha \in S$, $\Psi \in \Gamma$.
 But, in that case,
this set
has
a finite upper bound~[11]
that satisfies the inequality
$$
\sup_{(\alpha \in S, \Psi \in \Gamma)} I_\alpha(\Psi) \leq \frac{A(\alpha^*,
u^*)}{B(\alpha^*, u^*)} = I_{\alpha^*}(\Psi^*).
\eqno(21)
$$

 Since
$\alpha^* \in S$,
in view of the upper bound property~[11],
we can write
$$
I_{\alpha^*}(\Psi^*) \leq \sup_{\alpha \in S} I_\alpha(\Psi)
$$
for any fixed measure
$\Psi^* \in \Gamma^*$.

 In particular,
for
$\Psi^*=\Psi^*_{u^*}$, we have
$$
I_{\alpha^*}(\Psi^*_{u^*}) \leq \sup_{\alpha \in S} I_\alpha(\Psi^*_{u^*}).
\eqno(22)
$$

 At the same time,
$\Psi^*_{u^*} \in \Gamma^*$.
 Then, by the above-mentioned upper bound property,
we have
$$
\sup_{\alpha \in S} I_\alpha(\Psi^*_{u^*}) \leq \sup_{\alpha \in
S}\left[\sup_{\Psi^* \in \Gamma^*}I_\alpha(\Psi^*)\right].
\eqno(23)
$$

 Under the assumption of the theorem,
$\Gamma^* \subset \Gamma$.
 Then
by the upper bound property~[11],
for any fixed
$\alpha \in S$,
the following inequality holds:
$$
\sup_{\Psi^* \in \Gamma^*}I_\alpha(\Psi^*) \leq \sup_{\Psi \in
\Gamma}I_\alpha(\Psi),
$$
whence
$$
\sup_{\alpha \in S}\sup_{\Psi^* \in \Gamma^*}I_\alpha(\Psi^*) \leq \sup_{\alpha
\in S}\sup_{\Psi \in \Gamma}I_\alpha(\Psi) = \sup_{(\alpha \in S, \Psi \in
\Gamma)}I_\alpha(\Psi).
\eqno(24)
$$

 Note
additionally that the relations
$\alpha^* \in S$, $\Psi^*_{u^*} \in \Gamma* \subset \Gamma$
hold.
 Then,
by the definition of the upper bound,
$$
I_\alpha(\Psi^*_{u^*}) = \frac{A(\alpha^*, u^*)}{B(\alpha^*, u^*)} \leq
\sup_{(\alpha \in S, \Psi \in \Gamma)}I_\alpha(\Psi).
\eqno(25)
$$

 Using relation~(21), (22), (23), (24), (25) simultaneously,
we obtain
\begin{align}
\nonumber
I_\alpha(\Psi^*_{u^*}) &\leq
\sup_{\alpha \in S} I_\alpha(\Psi^*_{u^*}) \leq
\sup_{\alpha \in S}\left[\sup_{\Psi^* \in \Gamma^*}I_\alpha(\Psi^*)\right] \leq
\sup_{\alpha \in S}\sup_{\Psi \in \Gamma}I_\alpha(\Psi)
\\
&=\sup_{(\alpha \in S, \Psi \in \Gamma)}I_\alpha(\Psi) \leq
I_\alpha(\Psi^*_{u^*}).
\tag{26}
\end{align}

 It follows from relations~(26)
that the upper bound for the functional
$I_\alpha(\Psi)$
on the set
$(\alpha \in S$, $\Psi \in \Gamma)$
is attained
for
$\alpha=\alpha^*,
\Psi=\Psi^*$,
and the following equality holds:
$$
\sup_{(\alpha \in S, \Psi \in \Gamma)}I_\alpha(\Psi) = I_{\alpha^*}(\Psi^*_{u^*})
= \frac{A(\alpha^*, u^*)}{B(\alpha^*, u^*)}.
$$
 In this case, the upper bound
is the maximum of the functional under study
with respect to the set
$(\alpha \in S$, $\Psi \in \Gamma)$.

 Thus, we have
$$
\max_{(\alpha \in S, \Psi \in \Gamma)}I_\alpha(\Psi) = \max_{(\alpha \in S, \Psi^*
\in \Gamma^*)}I_\alpha(\Psi^*) = I_{\alpha^*}(\Psi^*_{u^*}) = \max_{(\alpha \in S, \Psi
\in \Gamma)} \frac{A(\alpha, u)}{B(\alpha, u)} = \frac{A(\alpha^*, u^*)}{B(\alpha^*,
u^*)}.
$$

 The first assertion of the theorem is proved.

 The second assertion is proved in a similar way.

 The proof of Theorem~1 is complete.
\end{proof}

 Now we shall study the solution of the extremal problem~(3)
for the variants in which the test function of the linear-fractional integral functional~(2)
does not attain its global extremum.

\begin{theorem}
 Suppose that the main objects in the extremal problem~(3)
satisfy conditions~\rom{1,~3} and the function
$B(\alpha, u)$
is strictly of constant sign
\rom(strictly positive or
strictly negative\rom)
 We also assume that the test function
$C(\alpha, u)$
is bounded
\rom(above or below\rom),
but does not attain its global extremum
\rom(maximum or minimum\rom)
on the set
$S \times U$.

 Then the following assertions hold\rom:
\begin{enumerate}
\item
 If the test function
$C(\alpha, u)$
is bounded above and does not attain its global maximum, then,
for any given
$\varepsilon > 0$,
there exists a pair
$\alpha^{(+)}(\varepsilon) \in S, u^{(+)} (\varepsilon) \in U$
such that the following inequality holds\rom:
$$
\sup_{(\alpha, \Psi) \in S \times
\Gamma}I_\alpha(\Psi) - \varepsilon <
I_{\alpha^{(+)}(\varepsilon)}\left(\Psi^*_{u^{(+)}(\varepsilon)}\right) =
\frac{A\left(\alpha^{(+)}(\varepsilon),
u^{(+)}(\varepsilon)\right)}{B\left(\alpha^{(+)}(\varepsilon),
u^{(+)}(\varepsilon)\right)}
< \sup_{(\alpha, \Psi) \in S \times
\Gamma}I_\alpha(\Psi),
\eqno(27)
$$
where
$\Psi^*_{u^{(+)}(\varepsilon)} \in \Gamma$
is a degenerate probability measure concentrated
at the point
$u^{(+)}(\varepsilon)$.

\item
 If the test function
$C(\alpha, u)$
is bounded below
and
does not attain its global minimum, then,
for any given
$\varepsilon > 0$,
there exists a pair
$\alpha^{(-)}(\varepsilon) \in S, u^{(-)} (\varepsilon) \in U$
such that the following inequality holds\rom:
$$
\inf_{(\alpha, \Psi) \in S \times
\Gamma}I_\alpha(\Psi) <
I_{\alpha^{(-)}(\varepsilon)}\left(\Psi^*_{u^{(-)}(\varepsilon)}\right) =
\frac{A\left(\alpha^{(-)}(\varepsilon),
u^{(-)}(\varepsilon)\right)}{B\left(\alpha^{(-)}(\varepsilon),
u^{(-)}(\varepsilon)\right)} <
< \inf_{(\alpha, \Psi) \in S \times
\Gamma}I_\alpha(\Psi) + \varepsilon,
\eqno(28)
$$
where
$\Psi^*_{u^{(-)}(\varepsilon)} \in \Gamma$
is a degenerate probability measure concentrated
at the point
$u^{(-)}(\varepsilon)$.
\end{enumerate}
\end{theorem}

 Note that assertions~1 and 2
may hold either
separately
for
the upper and lower bounds
or simultaneously for both bounds.

\begin{proof}[ \bf Proof of Theorem~2.]
 The following auxiliary statement holds.

\begin{lemma}
 Suppose that the assumptions of Theorem~\rom{1} hold\rom;
however, at the same time, the test function
$C(\alpha, u)$
is bounded above or below, but does not attain its global extremum
on the set of values of the arguments
$(\alpha, u) \in S \times
U$.
 Then the linear-fractional integral functional
$I_\alpha(\Psi)$
is also bounded above or below
and the following relations hold\rom:
$$
\sup_{(\alpha, \Psi) \in S \times
\Gamma}I_\alpha(\Psi) = \sup_{\alpha \in S}
\sup_{\Psi^* \in \Gamma^*}I_\alpha(\Psi^*)
= \sup_{(\alpha, u) \in S \times U}
\frac{A(\alpha, u)}{B(\alpha, u)} < \infty
\eqno(29)
$$
if the test function
$C(\alpha, u)$
is bounded above
or
$$
\inf_{(\alpha, \Psi) \in S \times
\Gamma}I_\alpha(\Psi) = \inf_{\alpha \in S}
\inf_{\Psi^* \in \Gamma^*}I_\alpha(\Psi^*)
= \inf_{(\alpha, u) \in S \times U}
\frac{A(\alpha, u)}{B(\alpha, u)} > - \infty
\eqno(30)
$$
if the test function
$C(\alpha, u)$
is bounded below.
\end{lemma}

 Relations~(29) and (30) given above
may hold either
separately
for
the upper and lower bounds of the values of the functional
$I_\alpha(\Psi)$,
or simultaneously for both bounds.

\begin{proof}[\bf Proof of Lemma~2]
 Suppose that the following condition holds:
$$
C(\alpha, u) = \frac{A(\alpha, u)}{B(\alpha, u)} \leq C^{(+)} < \infty, \qquad
(\alpha, u) \in S \times
U,
$$
but, at the same time, the function
$C(\alpha, u)$
does not attain its maximum value on the set
$S \times U$.
 As is well known if the set
is bounded above,
then it has a finite upper bound~[11]
and the following relation holds:
$$
C(\alpha, u) = \frac{A(\alpha, u)}{B(\alpha, u)} < \sup_{(\alpha, u) \in S
\times
U}\frac{A(\alpha, u)}{B(\alpha, u)} = C^{(+)}_0 \leq C^{(+)}, \qquad (\alpha, u) \in
S \times
U.
\eqno(31)
$$

 Using the assertion of Lemma~1,
from (31)
we obtain
$$
I_\alpha(\Psi) \leq C^{(+)}_0, \qquad \alpha \in S,\q \Psi \in \Gamma,
$$
whence,
by the upper bound property,
$$
\sup_{(\alpha, \Psi) \in S \times
\Gamma}I_\alpha(\Psi) \leq C^{(+)}_0.
\eqno(32)
$$

 At the same time,
by the property of the integral
$$
I_\alpha(\Psi^*_{\hat{u}}) = \frac{\int\limits_U A(\alpha,
u)\,d\Psi^*_{\hat{u}}(u)}{\int\limits_U B(\alpha, u)\,d\Psi^*_{\hat{u}}(u)} = \frac{A(\alpha,
\hat{u})}{B(\alpha, \hat{u})},
\eqno(33)
$$
where
$\hat{u} \in U$
is a fixed point and
$\Psi^*_{\hat{u}}(u)$
is a degenerate probability measure concentrated
at the point
$\hat{u}$.

 It follows from relation~(33) that
$$
\sup_{\alpha \in S}\sup_{\Psi^* \in \Gamma^*}I_\alpha(\Psi^*) = \sup_{\alpha
\in S}\sup_{u \in U}\frac{A(\alpha, u)}{B(\alpha, u)} = \sup_{(\alpha, u) \in S \times
U}\frac{A(\alpha, u)}{B(\alpha, u)} = C^{(+)}_0.
\eqno(34)
$$

 Since
$\Gamma^* \subset \Gamma$,
by the upper bound property
for any
fixed
$\alpha \in S$,
we have
$$
\sup_{\Psi^* \in \Gamma^*}I_\alpha(\Psi^*) \leq \sup_{\Psi \in
\Gamma}I_\alpha(\Psi),
$$
whence
$$
\sup_{\alpha \in S}\sup_{\Psi^* \in \Gamma^*}I_\alpha(\Psi^*) \leq
\sup_{\alpha \in S}\sup_{\Psi \in \Gamma}I_\alpha(\Psi).
\eqno(35)
$$

 Using relations~(32), (34), (35),
we can write
\begin{align*}
\sup_{(\alpha, \Psi) \in S \times
\Gamma}I_\alpha(\Psi) &\leq \sup_{(\alpha, u)
\in S \times
U}\frac{A(\alpha, u)}{B(\alpha, u)} = C^{(+)}_0 = \sup_{\alpha \in
S}\sup_{\Psi^* \in \Gamma^*}I_\alpha(\Psi^*)
\\
&\leq \sup_{\alpha \in S}\sup_{\Psi \in \Gamma}I_\alpha(\Psi) =
\sup_{(\alpha, \Psi) \in S \times
\Gamma}I_\alpha(\Psi)
\end{align*}
whence we directly obtain relation~(29).
 The second assertion of Lemma~2,
i.e., relation~(30) is proved in a similar way.
\end{proof}

 Let us now pass
to the direct proof of Theorem~2.
 We consider the variant
in which the test function
$C(\alpha, u)$
is bounded above, but does not attain its global maximum
on the set
$(\alpha, u) \in S \times U$.
 Then the first assertion of Lemma~2 holds,
i.e., relation~(29).
 Let us fix an arbitrary
$\varepsilon > 0$.
 Then,
by the upper bound property,
there exists an
$\alpha^{(+)}(\varepsilon) \in S$,
$u^{(+)}(\varepsilon) \in U$
such that the following inequality holds:
$$
\sup_{(\alpha, u) \in S \times
U}\frac{A(\alpha, u)}{B(\alpha, u)} - \varepsilon <
\frac{A\left(\alpha^{(+)}(\varepsilon),
u^{(+)}(\varepsilon)\right)}{B\left(\alpha^{(+)}(\varepsilon),
u^{(+)}(\varepsilon)\right)} < \sup_{(\alpha, u) \in S \times
U}\frac{A(\alpha,
u)}{B(\alpha, u)}.
\eqno(36)
$$

 Let
$\Psi^*_{u^{(+)}(\varepsilon)}$
denote a degenerate probability measure
concentrated at the point
$u^{(+)}(\varepsilon)$.
 Then,
by the property of the integral, we have
$$
I_{\alpha^{(+)}(\varepsilon)}\left(\Psi^*_{u^{(+)}(\varepsilon)}\right) =
\frac{A\left(\alpha^{(+)}(\varepsilon),
u^{(+)}(\varepsilon)\right)}{B\left(\alpha^{(+)}(\varepsilon),
u^{(+)}(\varepsilon)\right)}.
\eqno(37)
$$

 From relations~(29), (36), (37),
we obtain (27),
i.e., the first assertion of Theorem~2.
 The second assertion of this theorem
(relation~(28))
is proved in a similar way,
using the second assertion of Lemma~2.

 The proof of Theorem~2 is complete.
\end{proof}

\medskip

 The theoretical value of Theorem~2
consists in the fact that
if the test function of the linear-frac\-tional integral functional
is bounded, but does not attain its extremum, then,
for any given
$\varepsilon > 0$,
there exists an $\varepsilon$-optimal deterministic control strategy
determined by the values of
$\alpha^{(+)}(\varepsilon) \in S$,
$u^{(+)}(\varepsilon) \in U$
for the maximum problem or by the values of
$\alpha^{(-)}(\varepsilon) \in S, u^{(-)}(\varepsilon) \in U$
for the of minimum problem.

 Let us pass to the formulation and proof
of another statement related to the extremal problem
for a linear-fractional integral functional depending
on a parameter.

\begin{theorem}
 We assume that the main objects in the extremal problem~\rom{(3)}
satisfy the conditions~\rom{1,~3}
and the function
$B(\alpha, u)$
is of constant sign
(strictly
positive or
strictly
negative).
 We also assume that the test function
$C(\alpha, u)$
is not bounded
\rom(above or below\rom).
 Then the corresponding linear-fractional integral functional
is also not bounded above or below
and the following assertions hold:
\begin{enumerate}
\item
 There exists a sequence
$\left(\alpha^{(+)}_n, \Psi^{*(+)}_n\right)$, $\alpha^{(+)}_n \in S$, $\Psi^{*(+)}_n \in \Gamma^*$,
$n= 1, 2, ... $,
such that
$$
I\left(\alpha^{(+)}_n, \Psi^{*(+)}_n\right) \rightarrow \infty,\qq
n \rightarrow \infty,
\eqno(38)
$$
if the test function is not bounded above.

\item
 There exists a
a sequence
$\left(\alpha^{(-)}_n,\Psi^{*(-)}_n\right)$,
$\alpha^{(-)}_n \in S$, $\Psi^{*(-)}_n \in \Gamma^*$, $n= 1, 2, ...$,
such that
$$
I\left(\alpha^{(-)}_n, \Psi^{*(-)}_n\right) \rightarrow - \infty,\qq
n\rightarrow \infty,
\eqno(39)
$$
if the test function is not bounded below.

 If the test function
$C(\alpha, u)$
is bounded above
and does not attain its global maximum, then,
for any given
$\varepsilon > 0$,
there exists a pair
$\alpha^{(+)}(\varepsilon) \in S$,  $u^{(+)} (\varepsilon) \in U$
such that the inequality holds.
\end{enumerate}
\end{theorem}

\begin{proof} [\bf Proof]
 Suppose that the test function
$C(\alpha, u) ={A(\alpha, u)}/{B(\alpha, u)}$
is not bounded above.
 Then
there exists a sequence of pairs of points
$\left(\alpha^{(+)}_n, u^{(+)}_n\right)$, $\alpha^{(+)}_n \in S$,
$u^{(+)}_n \in U$, $n = 1, 2, ...$,
for which
$C\left(\alpha^{(+)}_n, u^{(+)}_n\right) \rightarrow \infty$,
$n \rightarrow \infty$.
 Consider the sequence of degenerate probability measures concentrated
at the points
$u^{(+)}_n: \Psi^{*(+)}_n = \Psi^*_{u^{(+)}_n}, n = 1, 2, ...$.
 By the property of the integral with respect to a degenerate measure,
we have
$$
I\left(\alpha^{(+)}_n, \Psi^{*(+)}_n\right) = C\left(\alpha^{(+)}_n,
u^{(+)}_n\right) \rightarrow \infty, n \rightarrow \infty,
$$
and so relation~(38) is proved.

 Relation~(39) can be established in a similar way.

 The proof of Theorem~3 is complete.
\end{proof}

 By Theorem~3, if the test function of the linear-fractional integral functional
is not bounded above or below,
then the solution of the corresponding unconditional extremum (maximum or minimum) problem
does not exist.

\section{Conclusions}

 In this paper, we have proved three theorems whose assertions
determine the solution of the unconditional extremal problem
for a linear-fractional integral functional depending
on a parameter.
 It is established that the solutions of the extremal problem
are completely defined by the properties of the test function.
 The results obtained
generalize statements from [1],~[2] on the unconditional extremum of a fractional linear integral functional
to the case where the test function is independent of the optimization parameter.
 The given results can be used to solve various applied stochastic control problems
in which the analytic structure of the objective functional
is described by a fractional linear integral functional.

\end{document}